\documentclass[12pt]{amsart}
\usepackage{amsmath, amsthm}

\newtheorem{theorem}{Theorem}[section]
\newtheorem{lemma}[theorem]{Lemma}

\theoremstyle{definition}
\newtheorem{definition}[theorem]{Definition}

\theoremstyle{remark}

\numberwithin{equation}{section}

\begin{document}

\title{Decomposability of Nonnegative $r$-Potent Matrices}

\author{Rashmi Sehgal Thukral}
\address{Department of Mathematics, Jesus and Mary College, University of Delhi, Chanakyapuri, New Delhi 110056}
\email{rashmi.sehgal@yahoo.co.in}

\author{Dr. Alka Marwaha}
\address{Department of Mathematics, Jesus and Mary College, University of Delhi, Chanakyapuri, New Delhi 110056}
\email{alkasamta@gmail.com}


\date{April 13, 2014.}


\begin{abstract}
We consider nonnegative $r$-potent matrices with finite dimensions and study their decomposability. We derive the precise conditions under which an $r$-potent matrix is decomposable. We further determine a general structure for the $r$-potent matrices based on their decomposibility. Finally, we establish that semigroups of $r$-potent matrices are also decomposable.
\end{abstract}

\maketitle

\section{Introduction}
A square matrix $\mathbf{A}$ is said to be idempotent \cite{HornAndJohnson} if and only if $\mathbf{A}^2=\textbf{A}$. The concept of $r$-potent matrices \cite{McCloskeyFirstPaper} is a generalization of idempotent matrices where a matrix $\mathbf{A}$ is said to be $r$-potent, for some natural number $r$, if and only if $\mathbf{A}^r=\mathbf{A}$. While every idempotent matrix is $r$-potent, the reverse is not necessarily true. That is, an $r$-potent matrix may or may not be idempotent. For example, the matrix
\begin{align}
&\left[\begin{array}{ccc}0& 1& 0\\ 0& 0& 1\\ 1& 0& 0\end{array}\right]
\end{align}
is $4$-potent (commonly known as quadripotent), but not idempotent. Therefore, it makes sense to study $r$-potent matrices separately. 

Several properties of $r$-potent matrices have been studied by McCloskey \cite{McCloskeyFirstPaper}. Decomposability of general $r$-potent matrices has however not been studied so far. On the other hand, decomposability of {\it idempotent} matrices has been studied by Marwaha \cite{AlkaMarwaha}. The tools developed in \cite{AlkaMarwaha} for idempotent matrices do not apply to $r$-potent matrices. The main focus of this paper is therefore to develop tools so that we can study decomposability of $r$-potent matrices. 

\subsubsection*{Outline of the paper:} The rest of the paper is organized as follows. We provide an overview of the known results used in this paper in Section~\ref{Sec:OverviewKnownResults}. In Section~\ref{Sec:DecomposabilitySingleMatrix}, we state our main result on decomposability of $r$-potent matrices and provide its  proof. We provide our results on the structure of $r$-potent matrices in Section~\ref{Sec:StructureRPotents}. In Section~\ref{Sec:DecomposabilityKronRPotents}, we establish  decomposability of the kronecker products of $r$-potent matrices. Section~\ref{Sec:DecomposabilitySemiGroupsRPotents} states our main results and corresponding proofs on decomposability of semigroups of $r$-potent matrices. Finally, we state our results on decomposability of permutation matrices in Section~\ref{Sec:DecomposabilityPermutationMatrices}.

\subsubsection*{Notation:} Capital letters $\mathbf{A},\mathbf{B},\cdots$ and small letters $\mathbf{a},\mathbf{b},\ldots$ are used to denote matrices and vectors, respectively, over $\mathbb{R}$, where $\mathbb{R}$ stands for the space of real numbers. For any set of vectors $\{\mathbf{v}_1,\mathbf{v}_2,\ldots\}$, $\vee\{\mathbf{v}_1,\mathbf{v}_2,\ldots\}$ denotes the (closed) linear span of the vectors $\{\mathbf{v}_1,\mathbf{v}_2,\ldots\}$. $\mathcal{M}_n(\mathbb{R})$ stands for the space of all $n\times n$ matrices with entries from $\mathbb{R}$. A matrix $\mathbf{A}=(a_{ij})\in\mathcal{M}_n(\mathbb{R})$ is called nonnegative if $a_{ij}\geq 0,\forall i,j=1,2,\ldots,n$. A nonnegative semigroup in $\mathcal{M}_n(\mathbb{R})$ is a semigroup of nonnegative matrices. Given two matrices $\mathbf{A}$ and $\mathbf{B}$, $\mathbf{A}\otimes\mathbf{B}$ denotes their kronecker product (\cite{Lutkepohl}, pp.3). Since every $n\times n$ real matrix represents a linear operator on $\mathbf{R}^n$ and vice versa (see \cite{Herstein}, pp.276), we use capital letters $\mathbf{A},\mathbf{B},\cdots$ to denote finite dimensional linear transformation on $\mathbb{R}^n$ and $n\times n$ real matrices interchangeably. $R(\mathbf{A})$ and $N(\mathbf{A})$ are used to denote the range space and the null space of linear transformation $\mathbf{A}$. Further, $\text{rank}(\mathbf{A})$  denotes the rank of $\mathbf{A}$ and $\text{Nullity}(\mathbf{A})$ denotes the dimension of the null space of $\mathbf{A}$. Finally, $\mathbf{A}^{-1}$ stands for the inverse of square matrix $\mathbf{A}$.

\section{Definitions and Overview of Known Results}
\label{Sec:OverviewKnownResults}

\begin{definition}\cite{AlkaMarwaha}
A matrix $\mathbf{A} \in\mathcal{M}_n(\mathbb{R})$ is said to be decomposable if there exists a proper subset  $\{i_1, i_2,\ldots,i_k\}$ of $\{1,2,\ldots,n\}$ such that $\vee\{\mathbf{A}\mathbf{e}_{i_1},\mathbf{A}\mathbf{e}_{i_2},\cdots,\mathbf{A}\mathbf{e}_{i_k}\}$ is contained in $\vee\{\mathbf{e}_{i_1},\mathbf{e}_{i_2},\ldots,\mathbf{e}_{i_k}\}$, where $\{\mathbf{e}_1,\mathbf{e}_2,\ldots,\mathbf{e}_n\}$ is the standard ordered basis of $\mathbb{R}^n$. 
\end{definition}

The following equivalent definition of decomposability, given as a proposition in \cite{AlkaMarwaha}, will be used throughout this paper:
\begin{definition}
\label{Def:SimplifiedDefinitionOfDecomposability}
A matrix $\mathbf{A}\in\mathcal{M}_n(\mathbb{R})$ is decomposable if and only if there exists a permutation matrix $\mathbf{P}$ such that 
\begin{align}
\mathbf{P}^{-1}\mathbf{A}\mathbf{P}&= \left[\begin{array}{cc} \mathbf{B}\ \ \mathbf{C}\\ \mathbf{0}\ \ \mathbf{D}\end{array}\right]
\end{align}
where $\mathbf{B}$ and $\mathbf{D}$ are square matrices. A matrix is said to be {\it indecomposable} if it is not decomposable. 
\end{definition}

The definition given above for decomposability of a single matrix is extended in the obvious manner to a semigroup in $\mathcal{M}_n(\mathbb{R})$ (\cite{RadjaviBook}, pp.104), where a common permutation matrix $\mathbf{P}$ decomposes every matrix in the semigroup. 

\begin{definition}(\cite{AlkaMarwahaThesis}, pp.7)
A linear operator $\mathbf{A}$ defined on an $n$-dimensional vector space $\mathcal{V}$ is {\it decomposable} if there exists a standard subspace (subspace spanned by a subset $\{\mathbf{e}_1,\mathbf{e}_2,\ldots,\mathbf{e}_k\}$ of standard basis vectors $\mathbf{e}_1,\mathbf{e}_2,\ldots,\mathbf{e}_n$) $\mathcal{M}$ of $\mathcal{V}$ such that $\mathcal{M}$ is invariant under the action of $\mathbf{A}$, that is, $\mathbf{A}(\mathcal{M})\subseteq \mathcal{M}$.
\end{definition}

\subsection{Properties of Nonnegative Matrices}
\label{Subsec:PropertiesOfNonnegativeMatrices}
Our focus in this paper is on decomposability of nonnegative matrices in $\mathcal{M}_n(\mathbb{R})$. For nonnegative matrices, we have the following properties (see \cite{HornAndJohnson}, pp.487-528, and \cite{CarlMeyer}, pp.661-687):
\begin{enumerate}
\item For every pair of indices $i$ and $j$, there exists a natural number $m$ such that $(\mathbf{A}^m)_{ij}$ is not equal to zero.
\item Period of an index $i$ is the greatest common divisor of all natural numbers $m$ such that $(\mathbf{A}^m)_{ii}>0$. If $\mathbf{A}$ is indecomposable, then period of every index is the same and is called the {\it period of $\mathbf{A}$}.
\item {\it Primitive Matrix}: A nonnegative matrix $\mathbf{A}$ is called primitive if its $m$-th power, $\mathbf{A}^m$, is positive for some natural number $m$ (that is, the same $m$ works for all pairs of indices). 
\item Primitive matrices are the same as indecomposable aperiodic nonnegative matrices. 
\end{enumerate}
Please note that, as a consequence of the above statements, every positive matrix is primitive and every primitive matrix is indecomposable. In other words, every positive matrix is indecomposable. Therefore, {\it we restrict our analysis in this paper to the study of decomposability of nonnegative matrices}.

\subsection{Perron-Forbenius Theorem for Indecomposable Nonnegative Matrices}
\begin{theorem}(see \cite{HornAndJohnson}, pp.487-528, and \cite{CarlMeyer}, pp.661-687)
\label{Th:PerronFrobeniusTheormForNonNegative}
Let $\mathbf{A}$ be an $n\times n$ nonnegative indecomposable matrix with period $h$ and spectral radius $\rho$. Then the following statements hold:
\begin{enumerate}
\item The number $\rho$ is a positive real number, is an eigenvalue of matrix $\mathbf{A}$, and is referred to as the Perron-Frobenius eigenvalue of $\mathbf{A}$. 
\item The Perron-Frobenius eigenvalue $\rho$ is simple.
\item Matrix $\mathbf{A}$ has exactly $h$ complex eigenvalues with absolute value $\rho$. Each one of these eigenvalues is a simple root of the characterisitic polynomial and is the product of $\rho$ with an $h$-th root of unity. 
\item If $\mathbf{A}$ is a nonnegative primitive matrix in $\mathcal{M}_n(\mathbb{R})$, then the square matrix $\mathbf{A}^{n^2-2n+2}$ is positive.
\end{enumerate}
\end{theorem}

\subsection{Decomposability of Nonnegative Idempotent Matrices}
\label{Subsection:DecomposabilityOfNonNegativeIdempotentMatrices}
We now summarize the results on decomposability of idempotent matrices \cite{AlkaMarwaha}:
\begin{enumerate}
\item Every nonnegative idempotent matrix of rank $k>1$ is decomposable.
\item For an idempotent matrix $\mathbf{A}$, $\text{trace}(\mathbf{A})=\text{rank}(\mathbf{A})$.
\item Let $\mathbf{A}$ be an $n\times n$ nonnegative idempotent matrix of rank $k>1$. Then, the following hold:
\begin{enumerate}
\item Any maximal standard block triangularization of $\mathbf{A}$ has the following two properties:
\begin{enumerate}
\item \label{item1} Each diagonal block is either zero or a positive idempotent matrix of rank one.
\item \label{item2} There are exactly $k$ non-zero diagonal blocks.
\end{enumerate}
\item There exists a standard block triangularization of $\mathbf{A}$ with the above stated properties \ref{item1} and \ref{item2} such that no two consecutive diagonal blocks are zero (so that the total number of diagonal blocks is less than or equal to $2k+1$).
\end{enumerate}
\item Suppose $\mathcal{S}$ is a band (semigroup of nonnegative idempotent matrices) in $\mathcal{M}_n(\mathbb{R})$ with nonnegative members such that $\text{rank}(\mathbf{S})>1$, for all $\mathbf{S}\in\mathcal{S}$. Then, $\mathcal{S}$ is decomposable.
\end{enumerate}

\subsection{}
The following lemma by Heydar Radjavi will be repeatedly used for proving our results in this paper:
\begin{lemma}(\cite{RadjaviBook}, pp.106)
\label{RepeatedlyUsedLemma}
Let $\mathbf{A}$ be a nonnegative matrix such that every positive power of $\mathbf{A}$ has at least one diagonal entry equal to zero. Then, $\mathbf{A}$ is decomposable and has $0$ as an eigenvalue. 
\end{lemma}

\subsection{Spectral Properties of $r$-Potent Matrices}\cite{McCloskeyFirstPaper}
Any $r$-potent matrix has $x^r-x$ as minimal polynomial. Therefore, eigenvalues of an invertible $r$-potent matrix are the $(r-1)$-th roots of unity. That is, the roots are $e^{\frac{2k\pi \iota}{r-1}},k=0,1,\ldots,r-2,$ when not counting multiplicities of the roots. For a singular $r$-potent matrix, zero lies in the spectrum apart from the aforesaid eigenvalues. 

\begin{lemma}\cite{McCloskeyFirstPaper}
\label{LemmaRankEqualsTrace}
For an $r$-potent matrix $\mathbf{A}$, $\text{rank}(\mathbf{A})=\text{trace}(\mathbf{A}^{r-1})$. 
\end{lemma}
\begin{proof}
We begin by noticing that $\mathbf{A}^{r-1}$ is an idempotent matrix because
\begin{align}
\label{Eq:IdempotencyOfRMinus1thPower}
\mathbf{A}^{2r-2}=\mathbf{A}^{r-2}\mathbf{A}^r=\mathbf{A}^{r-2}\mathbf{A}=\mathbf{A}^{r-1}.
\end{align} 
It then follows from Section~\ref{Subsection:DecomposabilityOfNonNegativeIdempotentMatrices} that $\text{trace}(\mathbf{A}^{r-1})=\text{rank}(\mathbf{A}^{r-1})$. In addition, since the range of $\mathbf{A}^r$ is contained in the range of $\mathbf{A}^{r-1}$, we get
\begin{align*}
\text{rank}(\mathbf{A}^r)\leq \text{rank}(\mathbf{A}^{r-1}) \leq \cdots \leq \text{rank}(\mathbf{A}) = \text{rank}(\mathbf{A}^r).
\end{align*}
This however implies that $\text{rank}(\mathbf{A}^r) = \text{rank}(\mathbf{A}^{r-1}) = \text{rank}(\mathbf{A})$,
which, in turn, yields $\text{rank}(\mathbf{A})=\text{rank}(\mathbf{A}^{r-1})=\text{trace}(\mathbf{A}^{r-1})$.
\end{proof}

\section{Decomposability of $r$-Potent Matrices}
\label{Sec:DecomposabilitySingleMatrix}

Before we analyse decomposability of an $r$-potent matrix, let us look at decomposability of idempotent matrices more closely. It was shown in \cite{AlkaMarwaha} (restated in Section~\ref{Subsection:DecomposabilityOfNonNegativeIdempotentMatrices} of this paper) that an idempotent matrix of $\text{rank} > 1$ is decomposable. An idempotent matrix of rank one, however, may or may not be decomposable. For example, the idempotent matrix 
\begin{align*}
\left(\begin{array}{cc}
1/2&1/2\\
1/2&1/2
\end{array}\right)
\end{align*}
is an {\it indecomposable} idempotent matrix of rank one, while 
\begin{align*}
\left(\begin{array}{cc}
1&0\\
0&0
\end{array}\right)
\end{align*}
is a {\it decomposable} idempotent matrix of rank one. Let us now extend the result in \cite{AlkaMarwaha} for idempotent matrices by stating a condition under which an idempotent matrix of rank one becomes decomposable:
\begin{theorem}
Let $\mathbf{A}$ be a nonnegative idempotent matrix of rank one. Then, $\mathbf{A}$ is decomposable if and only if it has at least one diagonal entry zero.
\end{theorem}
\begin{proof}
We first assume that $\mathbf{A}$ has at least one diagonal entry zero. Then, every positive power of $\mathbf{A}$ has at least one diagonal entry equal to zero because $\mathbf{A}^n=\mathbf{A},$ for all natural numbers $n$. Therefore, the result in \cite{RadjaviBook} (Lemma~\ref{RepeatedlyUsedLemma} of this paper) implies that $\mathbf{A}$ must be decomposable. 

We next assume that $\mathbf{A}$ is decomposable so that 
\begin{align}
\mathbf{P}^{-1}\mathbf{A}\mathbf{P}&=\left(\begin{array}{cc}
\mathbf{B}&\mathbf{C}\\
\mathbf{0}&\mathbf{D}
\end{array}\right)
\end{align}
for some permutation matrix $\mathbf{P}$. In addition, $\text{rank}(\mathbf{A})=1$ gives $\text{rank}(\mathbf{P}^{-1}\mathbf{A}\mathbf{P})=1$, which implies that either $\mathbf{B}=0$ or $\mathbf{D}=0$. It therefore follows that $\mathbf{A}$ has a zero diagonal entry. 
\end{proof}

Before proceeding to the statement of our main result on decomposability of nonnegative $r$-potent matrices in Theorem~\ref{Th:MainResultSection3}, two comments are in order:

 First, note that an $r$-potent matrix of $\text{rank}\leq r-1$ may or may not be decomposable. For example, if we define $\mathbf{A}$ as 
\begin{align*}
\mathbf{A}\mathbf{e}_1&=\mathbf{e}_2\\
\mathbf{A}\mathbf{e}_2&=\mathbf{e}_3\\
\vdots \quad &=\quad \vdots\\
\mathbf{A}\mathbf{e}_{r-1}&=\mathbf{e}_1,
\end{align*}
that is, we consider the matrix
\begin{align*}
\left[\begin{array}{cccc}
0&0&\cdots&1\\
1&0&\cdots&0\\
0&1&\cdots&0\\
\vdots&\vdots&\ddots&\vdots\\
0&0&\cdots&0
\end{array}\right],
\end{align*}
then $\mathbf{A}$ is an indecomposable $r$-potent matrix of rank $r-1$ as we cannot find a nontrivial standard invariant subspace. On the other hand, the matrix
\begin{align*}
\left[\begin{array}{cccc}
[1] &\mathbf{0}&\cdots&\mathbf{0}\\
\mathbf{0}&
\left[\begin{array}{cc}
1/2&1/2\\
1/2&1/2
\end{array}\right]
&\cdots&\mathbf{0}\\
\vdots&\vdots&\ddots&\vdots\\
\mathbf{0}&\mathbf{0}&\vdots&
\left[\begin{array}{ccc}
1/(r-1)&\cdots&1/(r-1)\\
1/(r-1)&\cdots&1/(r-1)\\
\vdots&\ddots&\vdots\\
1/(r-1)&\cdots&1/(r-1)
\end{array}\right]
\end{array}
\right]
\end{align*}
is a decomposable $r$-potent matrix of rank $r-1$. 

Second, it makes sense to analyse the decomposability of $r$-potent matrices of rank $>r-1$ (in addition to the decomposability of $r$-potent matrices of rank $\leq r-1$) because there exist an infinite number of such $r$-potent matrices. The existence of such matrices can be established using the properties of kronecker product. In particular, if $\mathbf{A}$ and $\mathbf{B}$ are two $r$-potent matrices of rank (say) $r-1$ each, then $\mathbf{A}\otimes\mathbf{B}$ would also be an $r$-potent matrix because: (see \cite{Lutkepohl}, pp.38)
\begin{align}
(\mathbf{A}\otimes\mathbf{B})^r&=\mathbf{A}^r\otimes\mathbf{B}^r=\mathbf{A}\otimes\mathbf{B}.
\end{align}
Moreover, (\cite{Lutkepohl}, pp.20)
\begin{align}
\text{rank}(\mathbf{A}\otimes\mathbf{B})&=\text{rank}(\mathbf{A})\cdot\text{rank}(\mathbf{B})=(r-1)^2>r-1.
\end{align}
Since the above analysis holds for any kronecker product, there exist an infinite number of $r$-potent matrices of rank $> r-1$.

To summarize, the above two observations imply that a nonnegative $r$-potent matrix of $\text{rank}\leq r-1$ may or may not be decomposable and there exists an infinite number of $r$-potent matrices of $\text{rank}> r-1$. We next state our main result on decomposability of nonnegative real $r$-potent matrices:

\begin{theorem}\label{Th:MainResultSection3}
For any nonnegative $r$-potent matrix $\mathbf{A}$,
\begin{enumerate}
\item If $\text{rank}(\mathbf{A})>r-1$, then $\mathbf{A}$ is always decomposable.
\item If $\text{rank}(\mathbf{A})\leq r-1$ such that $\mathbf{A}$ is singular and $\mathbf{A}^2,\mathbf{A}^3,\ldots,$ $ \mathbf{A}^{r-1}$ have at least one diagonal entry zero, then $\mathbf{A}$ is decomposable.
\end{enumerate}
\end{theorem}
\begin{proof}
We shall prove the two cases separately:
\begin{enumerate}
\item  Let us assume that $\mathbf{A}$ is indecomposable. Then, the Perron-Frobenius theorem (Theorem~\ref{Th:PerronFrobeniusTheormForNonNegative}) for indecomposable nonnegative matrices is applicable. This implies (substituting $\rho=1$) that the largest real positive eigenvalue of $\mathbf{A}$ (that is, $1$) is a simple root of the characteristic polynomial of $\mathbf{A}$. In addition, all other $(r-1)$-th roots of unity are simple roots of the characteristic polynomial. Moreover, $\mathbf{A}$ has no other eigenvalues. The above statements imply that $\text{rank}(\mathbf{A})=r-1$, which is a contradiction, and hence, $\mathbf{A}$ must be decomposable. \\

\item Let us again assume that $\mathbf{A}$ is indecomposable and consider three jointly exhaustive cases:
\begin{description}
\item[Case 1] $\text{rank}(\mathbf{A})=r-1$\\
Under the assumption of indecomposability, we can apply the Perron-Frobenius theorem (Theorem~\ref{Th:PerronFrobeniusTheormForNonNegative}) so that $1,\alpha,\alpha^2,$ $\ldots, \alpha^{r-2}$, where $\alpha =e^{\frac{2\pi\iota}{r-1}}$, are all simple roots of the characteristic polynomial of $\mathbf{A}$. Therefore,
\begin{align}
\label{Eq:TraceAsSumOfRootsOfUnity}
\text{trace}(\mathbf{A})&=1+\alpha +\alpha^2+\cdots +\alpha^{r-2}=0.
\end{align} 
Since $\mathbf{A}$ is nonnegative, we must have all the diagonal entries of $\mathbf{A}$ as nonnegative. However, this non-negativity, in conjunction with Eqn.~\ref{Eq:TraceAsSumOfRootsOfUnity}, implies that all the diagonal entries of $\mathbf{A}$ must be zero. Furthermore, since $\mathbf{A}$ is an $r$-potent matrix, this further implies that all the digaonal entries of 
\begin{align*}
\mathbf{A}&=\mathbf{A}^r=\mathbf{A}^{2r-1}=\mathbf{A}^{3r-2}=\cdots
\end{align*}
must be zero. Similarly, combining our condition that $\mathbf{A}^2,\mathbf{A}^3,\ldots, \mathbf{A}^{r-1}$ have at least one diagonal entry zero with the fact that $\mathbf{A}^r=\mathbf{A}$, we get that each of the following:
\begin{align*}
\mathbf{A}^2&=\mathbf{A}^{r+1}=\mathbf{A}^{2r}=\mathbf{A}^{3r-1}=\cdots\\
\mathbf{A}^3&=\mathbf{A}^{r+2}=\mathbf{A}^{2r+1}=\mathbf{A}^{3r}=\cdots\\
\vdots\\
\mathbf{A}^{r-1}&=\mathbf{A}^{2r-2}=\mathbf{A}^{3r-3}=\mathbf{A}^{4r-4}=\cdots
\end{align*}
have at least one diagonal entry zero. This, thanks to Lemma~\ref{RepeatedlyUsedLemma}, however implies  that $\mathbf{A}$ must be decomposable, which is a contradiction. \\

\item[Case 2] $\text{rank}(\mathbf{A})=1$\\
As $\mathbf{A}$ is indecomposable, we can again apply the Perron Frobenius theorem (Theorem~\ref{Th:PerronFrobeniusTheormForNonNegative}) so that $1$, being the Perron Frobenius eigenvalue, is  a simple eigenvalue. Moreover, $\text{rank}(\mathbf{A})=1$ implies that there is no other eigenvalue with absolute value $1$. Furthermore, the number of such eigenvalues is always equal to the period of the matrix. Therefore, we must have $\mathbf{A}$ to be aperiodic and hence primitive (Section~\ref{Subsec:PropertiesOfNonnegativeMatrices}). Now, for every primitive matrix $\mathbf{A}$ of order $n$, $\mathbf{A}^{n^2-2n+2}$ is a positive matrix (from Theorem~\ref{Th:PerronFrobeniusTheormForNonNegative}). Finally, since $\mathbf{A}$ is singular and $\text{rank}(\mathbf{A})=1$, the matrix $\mathbf{A}$ must be a square matrix of size $2\times 2$ or higher. We can now argue that 
\begin{align*}
\text{for }n=2,&\quad \mathbf{A}^{n^2-2n+2}=\mathbf{A}^2\ \text{is positive,}\\
\text{for }n=3,&\quad \mathbf{A}^{n^2-2n+2}=\mathbf{A}^5\ \text{is positive,}\\
\text{for }n=4,&\quad \mathbf{A}^{n^2-2n+2}=\mathbf{A}^{10}\ \text{is positive,}\\
&\quad \vdots
\end{align*}
which is a contradiction to the statement of this theorem that $\mathbf{A}^2,\mathbf{A}^3,\ldots,$ $ \mathbf{A}^{r-1}$ have at least one diagonal entry zero.

\item[Case 3] $1<\text{rank}(\mathbf{A})<r-1$\\
Let $\text{rank}(\mathbf{A})=k,$  where $1<k<r-1$. The number of eigenvalues of $\mathbf{A}$ would then be equal to the $\text{rank}(\mathbf{A})=k$. However, the assumed indecomposability of $\mathbf{A}$ implies that the number of eigenvalues of $\mathbf{A}$ is also equal to the period of $\mathbf{A}$. Hence, the period of $\mathbf{A}$ must be equal to $k$. This is however a contradiction because $\mathbf{A}^k$ has at least one diagonal element as zero due to the condition of the theorem, while $\mathbf{A}^k$ should be a positive matrix according to Theorem~\ref{Th:PerronFrobeniusTheormForNonNegative}.
\end{description}
Therefore, in each of the above jointly exhaustive cases, we have a contradiction to our assumption that $\mathbf{A}$ is indecomposable. This concludes the proof that $\mathbf{A}$ must be decomposable under the conditions stated in the theorem. 
\end{enumerate}
\end{proof}

\section{Structure of an $r$-Potent Matrix}
\label{Sec:StructureRPotents}
In this section, we study the structure of decomposable nonnegative $r$-potent matrices. Since a decomposable matrix can always be written in a block triangular form (by Defn.~\ref{Def:SimplifiedDefinitionOfDecomposability}) via  a permutation matrix, our focus will be on the properties of the diagonal blocks in such a block triangular form. However, before proceeding to our results on the properties of these diagonal blocks, we briefly digress to state a few comments (adopted from \cite{AlkaMarwaha}) on block triangularization that would be useful in proving our results in this section.

Let $\mathcal{S}$ be a semigroup of matrices in $\mathcal{M}_n(\mathbb{R})$ and $\mathcal{L}at'\mathcal{S}$ be the lattice of all standard subspaces which are invariant under every member of $\mathcal{S}$. It can be shown by simple induction that for any semigroup $\mathcal{S}$, $\mathcal{L}at'\mathcal{S}$ has a maximal chain of such subspaces. This chain may be non-trivial or trivial according to whether $\mathcal{S}$ has a non-trivial standard subspace or not. Each nontrivial chain in $\mathcal{L}at'\mathcal{S}$ gives rise to a block triangularization for $\mathcal{S}$ and since the members in the chain are standard subspaces, we shall call it a standard block triangularization. Evidently, to say that $\mathcal{S}$ has a standard block triangularization is equivalent to saying that there exists a permutation matrix $\mathbf{P}$ such that for each $\mathbf{S}\in\mathcal{S}$, $\mathbf{P}^{-1}\mathbf{S}\mathbf{P}$ has the upper block triangular form. Suppose $\mathcal{C}$ is a chain in $\mathcal{L}at'\mathcal{S}$ and $\mathcal{N}_1$ and $\mathcal{N}_2$ are two successive elements in $\mathcal{C}$ such that $\mathcal{N}_2\subseteq\mathcal{N}_1$, then $\mathcal{N}_1\ominus\mathcal{N}_2$ is called a gap in the chain. If $\mathbf{P}$ is the orthogonal projection onto $\mathcal{N}_1\ominus\mathcal{N}_2$, then the restriction of $\mathbf{P}\mathcal{S}\mathbf{P}$ to the range of $\mathbf{P}$ is called the compression of $\mathcal{S}$ to $\mathcal{N}_1\ominus\mathcal{N}_2$. Every such compression corresponds to a diagonal block in the block triangularization of $\mathbf{S}$. 

We now state our main result in this section:

\begin{theorem}
Any maximal standard block triangularisation of a decomposable nonnegative $r$-potent matrix has the following properties:
\begin{enumerate}
\item Each diagonal block is either zero or an indecomposable $r$-potent matrix of $\text{rank}\leq r-1$.
\item If $n$ is the number of non-zero diagonal blocks, then 
\begin{align}
\frac{k}{r-1}&\leq n\leq k,
\end{align}
where $k=\text{rank}(\mathbf{A})$.
\item Total number of diagonal blocks (including $\mathbf{0}$ blocks) lies between $\frac{k}{r-1}$ and $2k+1$.
\end{enumerate}
\end{theorem}

\begin{proof} We shall prove the above three statements in order:
\begin{enumerate}
\item Let 
\begin{align}
\mathbf{P}^{-1}\mathbf{A}\mathbf{P}&=
\left(\begin{array}{cccc}
\mathbf{A}_{11}&\mathbf{*}&\cdots&\mathbf{*}\\
\mathbf{0}&\mathbf{A}_{22}&\cdots&\vdots\\
\vdots&\vdots&\ddots&\vdots\\
\mathbf{0}&\cdots&\cdots&\mathbf{A}_{nn} \end{array}
\right)
\end{align}
be a maximal standard block triangularization of $\mathbf{A}$ via a permutation matrix $\mathbf{P}$. As $\mathbf{A}^r=\mathbf{A}$, we have 
$(\mathbf{P}^{-1}\mathbf{A}\mathbf{P})^r=\mathbf{P}^{-1}\mathbf{A}^r\mathbf{P}=\mathbf{P}^{-1}\mathbf{A}\mathbf{P}$ 
so that $\mathbf{A}_{ii}^r=\mathbf{A}_{ii}$, for all $i=1,2,\ldots,n$.  Thus, each diagonal block is itself an $r$-potent matrix. Further, each such diagonal block $\mathbf{A}_{ii}$ would be indecomposable and would have $\text{rank}(\mathbf{A}_{ii})\leq r-1$. This can be seen as follows: Suppose some $\mathbf{A}_{jj}$ is decomposable with standard subspace $\mathcal{K}$. Now, $\mathbf{A}_{jj}$ corresponds to some gap, say $\mathcal{N}_1\ominus\mathcal{N}_2$, in the maximal chain of invariant subspaces for the aforesaid block triangularization of $\mathbf{A}$. Then, $\mathcal{N}_2\oplus\mathcal{K}$ is a standard subspace, invariant under $\mathbf{A}$ which lies strictly between $\mathcal{N}_1$ and $\mathcal{N}_2$, thus contradicting the maximality of the above triangularization. In other words, $\mathbf{A}_{jj}$ must be indecomposable. Finally, since all $r$-potent matrices of $\text{rank} > r-1$ are decomposable by Theorem~\ref{Th:MainResultSection3}, we get $\text{rank}(\mathbf{A}_{ii})\leq r-1,\forall i$. \\
 
\item We start by noticing that
\begin{align}
\text{trace}(\mathbf{A}^{r-1})&=\text{trace}(\mathbf{A}_{11}^{r-1})+\cdots+\text{trace}(\mathbf{A}_{nn}^{r-1}).
\end{align}
Therefore,
\begin{align}
k&=\text{rank}(\mathbf{A})\\
&=\text{trace}(\mathbf{A}^{r-1})\qquad\text{(from Lemma~\ref{LemmaRankEqualsTrace})}\\
&=\text{trace}(\mathbf{A}_{11}^{r-1})+\cdots+\text{trace}(\mathbf{A}_{nn}^{r-1})\\
&=\text{rank}(\mathbf{A}_{11})+\cdots+\text{rank}(\mathbf{A}_{nn})\qquad\text{(from Lemma~\ref{LemmaRankEqualsTrace})}\\
&\leq (r-1)+\cdots+(r-1)\\
&=n(r-1),
\end{align}
and therefore, we get the lower bound:
\begin{align}
n&\geq\frac{k}{r-1}.
\end{align}
On the other hand, let us consider an extreme case that each non-zero diagonal block has rank $1$. Then, the maximum number of non-zero diagonal blocks is $k$, so that we get the upper bound $n\leq k$. Combining the two bounds, we can write
\begin{align}
\frac{k}{r-1}\leq n\leq k.
\end{align}

\item We claim that two consecutive diagonal blocks cannot be zero. Suppose that two consecutive diagonal blocks are zero. Then, a $2\times 2$ block matrix 
\begin{align*}
\left(\begin{array}{cc} \mathbf{0}&\mathbf{B}\\ \mathbf{0}&\mathbf{0}\end{array}\right) 
\end{align*}
would be an $r$-potent if and only if it is zero. Therefore, the total number of diagonal blocks (including $\mathbf{0}$ blocks) lies between $\frac{k}{r-1}$ and $2k+1$.
\end{enumerate}
\end{proof}

\section{Decomposability of Kronecker Product of $r$-potent matrices}
\label{Sec:DecomposabilityKronRPotents}

In this section, we shall discuss the decomposability of kronecker products of $r$-potent matrices. We start by reiterating that the kronecker product of two $r$-potent matrices is itself an $r$-potent matrix because
\begin{align}
(\mathbf{A}\otimes\mathbf{B})^r&=\mathbf{A}^r\otimes\mathbf{B}^r\\
&=\mathbf{A}\otimes\mathbf{B}.
\end{align} 
We can now state the two main results in this section as follows:
\begin{theorem}
Let $\mathbf{A}$ be a nonnegative $r$-potent matrix of $\text{rank}>r-1$ and $\mathbf{B}$ be any non-zero nonnegative $r$-potent matrix. Then, $\mathbf{A}\otimes\mathbf{B}$ is a decomposable $r$-potent of $\text{rank}>r-1$.
\end{theorem}
\begin{proof}
Since
\begin{align}
\text{rank}(\mathbf{A}\otimes\mathbf{B})&=\text{rank}(\mathbf{A})\text{rank}(\mathbf{B})\\
&>(r-1)\text{rank}(\mathbf{B}),
\end{align}
which implies that
\begin{align}
\text{rank}(\mathbf{A}\otimes\mathbf{B})&> r-1.
\end{align}
Therefore, from Theorem~\ref{Th:MainResultSection3} of this paper, $\mathbf{A}\otimes\mathbf{B}$ is a decomposable $r$-potent matrix.
\end{proof}

\begin{theorem}
Let $\mathbf{A}$ be a nonnegative $r$-potent matrix of $\text{rank}>r-1$ and $\mathbf{B}$ be a non-zero nonnegative  idempotent matrix. Then, $\mathbf{A}\otimes\mathbf{B}$ is a decomposable $r$-potent matrix.
\end{theorem}
\begin{proof}
Since an idempotent matrix is, by definition, also an $r$-potent matrix, we have
\begin{align}
(\mathbf{A}\otimes\mathbf{B})^r&=\mathbf{A}^r\otimes\mathbf{B}^r=\mathbf{A}\otimes\mathbf{B},
\end{align}
and therefore, $\mathbf{A}\otimes\mathbf{B}$ is also an $r$-potent matrix. In addition,  
\begin{align}
\text{rank}(\mathbf{A}\otimes\mathbf{B})&=\text{rank}(\mathbf{A})\text{rank}(\mathbf{B})\\
&>(r-1)\text{rank}(\mathbf{B}),
\end{align}
which implies that
\begin{align}
\text{rank}(\mathbf{A}\otimes\mathbf{B})&> r-1,
\end{align}
and hence, $\mathbf{A}\otimes\mathbf{B}$ is decomposable by Theorem~\ref{Th:MainResultSection3} of this paper. 
\end{proof}

\section{Decomposibility of Semi-Groups of $r$-Potent Matrices}
\label{Sec:DecomposabilitySemiGroupsRPotents}

Let us first analyse the decomposability of a semigroup generated by a single $r$-potent matrix. 
\begin{theorem}
Let $\mathbf{A}$ be a nonnegative $r$-potent matrix. Consider the semigroup $\mathcal{S}=\{\mathbf{A},\mathbf{A}^2,\cdots,\mathbf{A}^{r-1}\}$ generated by $\mathbf{A}$.
\begin{enumerate}
\item If $\text{rank}(\mathbf{A}) > r-1$, then $\mathcal{S}$ is decomposable. 
\item If $\text{rank}(\mathbf{A})\leq r-1$, $\mathbf{A}$ is singular and $\mathbf{A}^2,\mathbf{A}^3,\cdots,\mathbf{A}^{r-1}$ have a zero diagonal entry, then $\mathcal{S}$ is decomposable.
\end{enumerate}
\end{theorem}

\begin{proof}
We shall prove the two parts seperately.
\begin{enumerate}
\item \label{SimpleSemigroupFiniteEasyCase}
Since $\text{rank}(\mathbf{A})>r-1$, $\mathbf{A}$ is decomposable by Theorem~\ref{Th:MainResultSection3}. Let
\begin{align}
\mathbf{P}^{-1}\mathbf{A}\mathbf{P}&=
\left(\begin{array}{cc}
\mathbf{B}&\mathbf{C}\\
\mathbf{0}&\mathbf{D}
\end{array}\right)
\end{align}
be a block-triangular decomposition of $\mathbf{A}$. Then
\begin{align}
(\mathbf{P}^{-1}\mathbf{A}\mathbf{P})^k&=\mathbf{P}^{-1}\mathbf{A}^k\mathbf{P}=
\left(\begin{array}{cc}
\mathbf{B}^k&*\\
\mathbf{0}&\mathbf{D}^k
\end{array}\right)
\end{align}
for all $k=2,3,\ldots,r-1$. Therefore, $\mathcal{S}$ is decomposable via permutation matrix $\mathbf{P}$.

\item Under the stated conditions, $\mathbf{A}$ is decomposable. Decomposability of $\mathcal{S}$ follows from an argument similar to (\ref{SimpleSemigroupFiniteEasyCase}) above. 
\end{enumerate}
\end{proof}

The above theorem motivates us to study decomposability of a general semigroup of $r$-potent matrices. To this end, we require the following equivalent conditions for decomposability of nonnegative semigroups in $\mathcal{M}_n(\mathbb{R})$.

\begin{lemma} \cite{AlkaMarwaha}
\label{LemmaOnEquivalenceOfSemigroupDecomposability}
For a semigroup $\mathcal{S}$ in $\mathcal{M}_n(\mathbb{R})$ with nonnegative matrices, the following are equivalent:
\begin{enumerate}
\item $\mathcal{S}$ is decomposable.
\item There exists a non-zero nonnegative functional on $\mathcal{M}_n(\mathbb{R})$ whose restriction to $\mathcal{S}$ is zero.
\item $\mathcal{S}$ has a common zero entry.
\item $\mathcal{S}$ has a common non-diagonal zero entry.
\item There exist $\mathbf{A},\mathbf{B}\in\mathcal{M}_n(\mathbb{R})$, both non-zero and nonnegative such that $\mathbf{A}\mathcal{S}\mathbf{B}=\{0\}$.
\end{enumerate}
\end{lemma}

We now proceed to the two main results of this section and detail their respective proofs:
\begin{theorem}
\label{FiniteDimensionDecomposabilityOfSimpleSemigroup}
Let $\mathcal{S}$ be a semigroup of nonnegative $r$-potent matrices of $\text{rank}> r-1$. Then, $\mathcal{S}$ is decomposable. 
\end{theorem}
\begin{proof}
We start by noticing that a semigroup of $r$-potent matrices always contains an idempotent matrix. For example, if $\mathbf{A}$ is an $r$-potent matrix in $\mathcal{S}$, then $\mathbf{A}^{r-1}$ is an idempotent matrix in $\mathcal{S}$ (see Eqn.~\ref{Eq:IdempotencyOfRMinus1thPower}). Consider now a minimal rank idempotent matrix $\mathbf{P}$ in $\mathcal{S}$. We can always choose such a minimal rank idempotent matrix because if we choose any minimal rank matrix $\mathbf{B}$ in $\mathcal{S}$, then the rank of the corresponding idempotent matrix $\mathbf{B}^{r-1}$ is upperbounded by $\text{rank}(\mathbf{B})$ because (\cite{Lutkepohl}, pp.61) 
\begin{align}
\text{rank}(\mathbf{B}^{r-1})
&\leq \text{min}\{\text{rank}(\mathbf{B}^{r-2}),\text{rank}(\mathbf{B})\}\\
&\leq\text{rank}(\mathbf{B}).
\end{align} 
Now, since $\text{rank}(\mathbf{P})>r-1>1$, $\mathbf{P}$ must be decomposable. Let $\mathbf{P}$ have the form
\begin{align*}
\left(\begin{array}{cc}
\mathbf{P}_1&\mathbf{K}\\
\mathbf{0}&\mathbf{P}_2
\end{array}
\right)
\end{align*}
with respect to some permutation of basis where both $\mathbf{P}_1$ and $\mathbf{P}_2$ are non-zero.

Let us now consider an arbitrary $\mathbf{S}\in\mathcal{S}$. Then, $(\mathbf{P}\mathbf{S}\mathbf{P})^{r-1}$ will be an idempotent matrix, also in $\mathcal{S}$. Further, the range of  $(\mathbf{P}\mathbf{S}\mathbf{P})^{r-1}$ is contained in the range of $\mathbf{P}$ and the null space of $(\mathbf{P}\mathbf{S}\mathbf{P})^{r-1}$ contains the null space of $\mathbf{P}$. Therefore, 
\begin{align}
\text{rank}((\mathbf{P}\mathbf{S}\mathbf{P})^{r-1})\leq\text{rank}(\mathbf{P}).
\end{align}
Considering that $\mathbf{P}$ is minimal rank in $\mathcal{S}$, we get
\begin{align}
\text{rank}((\mathbf{P}\mathbf{S}\mathbf{P})^{r-1})=\text{rank}(\mathbf{P}).
\end{align}
 Using rank-nullity theorem of linear algebra (\cite{CarlMeyer}, pp.199), we have nullity of $(\mathbf{P}\mathbf{S}\mathbf{P})^{r-1}$ is equal to the nullity of $\mathbf{P}$, which, in turn, yields
\begin{align}
 (\mathbf{P}\mathbf{S}\mathbf{P})^{r-1}&=\mathbf{P}.
\end{align}
This, however,  implies that
\begin{align}
\mathbf{P}\mathbf{S}\mathbf{P}&=\mathbf{P}^{\frac{1}{r-1}}=\mathbf{P}
\end{align}
 due to the following lemma:

\begin{lemma}
In the given semigroup $\mathcal{S}$, the only nonnegative $(r-1)$-th root of minimal rank idempotent $\mathbf{P}$ is $\mathbf{P}$ itself. 
\end{lemma}
\begin{proof}
Since 
\begin{align}
\label{Eq:PisIdempotent}
\mathbf{P}^2&=\mathbf{P}\\
\Rightarrow \mathbf{P}^{r-1}&=\mathbf{P},
\end{align}
$\mathbf{P}$ is itself a nonnegative $(r-1)$-th root of $\mathbf{P}$ in $\mathcal{S}$. Let $\mathbf{P}'$ be another nonnegative $(r-1)$-th root of $\mathbf{P}$ in $\mathcal{S}$, that is, 
\begin{align}
\label{Eq:PdashPowerR-1EqualP}
(\mathbf{P}')^{r-1}&=\mathbf{P}.
\end{align}
As $\mathbf{P}'$ belongs to $\mathcal{S}$, we also have 
\begin{align}
\label{Eq:RPotencyOfPdash}
(\mathbf{P}')^r&=\mathbf{P}'.
\end{align} 
It follows from  Eqn.~\ref{Eq:PdashPowerR-1EqualP} and Eqn.~\ref{Eq:RPotencyOfPdash} that
\begin{align}
\mathbf{P}\mathbf{P}'=\mathbf{P}'\mathbf{P}&=\mathbf{P}'.
\end{align}
Using the fact that $\mathbf{P}$ is of minimal rank in $\mathcal{S}$, it follows from rank-nullity theorem that $R(\mathbf{P})=R(\mathbf{P'})$ and $N(\mathbf{P})=N(\mathbf{P'})$. Moreover, Eqn.~\ref{Eq:PisIdempotent} implies that \begin{align}
\mathbf{P} &=\mathbf{I}\qquad \text{on}\quad R(\mathbf{P}).
\end{align}
which, due to Eqn.~\ref{Eq:PdashPowerR-1EqualP}, implies that 
\begin{align}
(\mathbf{P'})^{r-1}=\mathbf{I}\qquad \text{on}\quad R(\mathbf{P})
\end{align}
where $\mathbf{I}$ is the identity operator. In other words, $\mathbf{P'}$ is a nonnegative $(r-1)$-th root of $\mathbf{I}$ on $R(\mathbf{P})$. However, the only nonnegative $(r-1)$-th root of the identity operator is the identity operator itself. Therefore
\begin{align}
\mathbf{P'}=\mathbf{I}=\mathbf{P} \qquad\text{on}\quad R(\mathbf{P})=R(\mathbf{P'})
\end{align} 
and
\begin{align}
\mathbf{P'}=\mathbf{P}=\mathbf{0}\qquad \text{on}\quad N(\mathbf{P})=N(\mathbf{P'}).
\end{align}
Therefore
\begin{align}
\mathbf{P'}&=\mathbf{P}.
\end{align}
\end{proof}

Finally, let 
\begin{align*}
\left(\begin{array}{cc}
\mathbf{S}_{11}&\mathbf{S}_{12}\\
\mathbf{S}_{21}&\mathbf{S}_{22}
\end{array}\right)
\end{align*}
be the representation of an arbitrary $\mathbf{S}\in\mathcal{S}$ with respect to this permuted basis. Then, $\mathbf{P}\mathbf{S}\mathbf{P}=\mathbf{P}$ implies that $\mathbf{P}_2\mathbf{S}_{21}\mathbf{P}_1=\mathbf{0}$. If $p_{ij}$ and $r_{kl}$ are non-zero entries in $\mathbf{P}_2$ and $\mathbf{P}_1$, respectively, then it is easy to see that the $(j,k)$-th entry in each $\mathbf{S}\in\mathcal{S}$ is zero. This makes use of the fact that $\mathbf{P}_2,\mathbf{S}_{21}$, and $\mathbf{P}_1$ are all nonnegative matrices. By Lemma~\ref{LemmaOnEquivalenceOfSemigroupDecomposability}, $\mathcal{S}$ is decomposable.
\end{proof}

\begin{theorem}
Let $\mathcal{S}$ be a semigroup of nonnegative $r$-potent matrices of $\text{rank}>r-1$. Then, any maximal standard block triangularization of $\mathcal{S}$ has the property that each non-zero diagonal block is a semigroup of nonnegative $r$-potent matrices with at least one element of $\text{rank}\leq r-1$. 
\end{theorem}
\begin{proof}
By Theorem~\ref{FiniteDimensionDecomposabilityOfSimpleSemigroup}, $\mathcal{S}$ is decomposabable. Consider any maximal chain in $\mathcal{L}at'\mathcal{S}$ resulting in a standard block triangularization of $\mathcal{S}$. Consider any two subspaces $\mathcal{N}_1$ and $\mathcal{N}_2$ in this chain such that $\mathcal{N}_1\subseteq\mathcal{N}_2$ and $\mathcal{N}_2\ominus\mathcal{N}_1$ is a gap. If the compression of $\mathcal{S}$ to $\mathcal{N}_2\ominus\mathcal{N}_1$ is non-zero, it forms a semigroup of nonnegative $r$-potents. Further, it must be indecomposable, for otherwise, if it has a standard invariant subspace $\mathcal{K}$, then $\mathcal{N}_1\oplus\mathcal{K}$ is in $\mathcal{L}at'\mathcal{S}$ and lies strictly between $\mathcal{N}_1$ and $\mathcal{N}_2$, contradicting the maximality of this chain. Thus, every non-zero compression (or diagonal-block) constitutes an indecomposable semigroup of $r$-potent matrices. By Theorem~\ref{FiniteDimensionDecomposabilityOfSimpleSemigroup}, it must contain at least one element of $\text{rank}\leq r-1$.
\end{proof}

\section{Decomposability of permutation matrices}
\label{Sec:DecomposabilityPermutationMatrices}
In this section, we shall study decomposability (or rather, indecomposability) of permutation matrices. We start by noticing that an $n\times n$ circulant matrix generated by $\mathbf{e}_1,\mathbf{e}_2,\ldots,\mathbf{e}_n$ (standard basis vectors), given by
\begin{align*}
\left[\begin{array}{cccccc}
0&0&0&0&\cdots&1\\
0&1&0&0&\cdots&0\\
0&0&1&0&\cdots&0\\
\vdots&\vdots&\vdots&\vdots&\ddots&\vdots\\
1&0&0&0&\cdots&0
\end{array}
\right]_{n\times n},
\end{align*}
is indecomposable as we cannot find a standard invariant subspace. To generalize this observation as a formal result, we require the following lemma:
\begin{lemma}(\cite{RadjaviBook}, pp.106)
\label{LemmaDecomposabilityOfPermutationMatrices}
For a semigroup $\mathcal{S}$ of nonnegative matrices, the following are equivalent:
\begin{enumerate}
\item $\mathcal{S}$ is decomposable.
\item Every sum of members of $\mathcal{S}$ has a zero entry.
\end{enumerate}
\end{lemma}

We now state the main result of this section.
\begin{theorem}
Group $\mathcal{G}$ of $n\times n$ permutation matrices generated by the standard basis vectors $\mathbf{e}_1,\mathbf{e}_2,\ldots,\mathbf{e}_n$, comprises of idempotent matrices, tripotent matrices, quadripotent matrices, pentapotent matrices,$\ldots, (n+1)$-potent matrices and is indecomposable.
\end{theorem}
\begin{proof}
The group comprises the following:
\begin{description}
\item[Idempotent Matrices]
\begin{align*}
\mathbf{I}_{n\times n}&=
\left[\begin{array}{cccccc}
1&0&0&0&\cdots&0\\
0&1&0&0&\cdots&0\\
0&0&1&0&\cdots&0\\
\vdots&\vdots&\vdots&\vdots&\ddots&\vdots\\
0&0&0&0&\cdots&1
\end{array}
\right]_{n\times n}
\end{align*}
\item[Tripotent Matrices] are formed by permuting two rows at a time. For example  
\begin{align*}
\mathbf{A}_{21}&=
\left[\begin{array}{cccccc}
0&1&0&0&\cdots&0\\
1&0&0&0&\cdots&0\\
0&0&1&0&\cdots&0\\
\vdots&\vdots&\vdots&\vdots&\ddots&\vdots\\
0&0&0&0&\cdots&1
\end{array}
\right]_{n\times n}
\end{align*}
\item[Quadripotent Matrices] are formed by permuting three rows at a time. For example  
\begin{align*}
\mathbf{A}_{231}&=
\left[\begin{array}{cccccc}
0&1&0&0&\cdots&0\\
0&0&1&0&\cdots&0\\
1&0&0&0&\cdots&0\\
0&0&0&1&\cdots&0\\
\vdots&\vdots&\vdots&\vdots&\ddots&\vdots\\
0&0&0&0&\cdots&1
\end{array}
\right]_{n\times n}
\end{align*}

\item[(n+1)-potent Matrices] are formed by permuting all $n$ rows at a time. For example  
\begin{align*}
\mathbf{A}_{234\ldots (n+1)1}&=
\left[\begin{array}{cccccc}
0&1&0&0&\cdots&0\\
0&0&1&0&\cdots&0\\
\vdots&\vdots&\vdots&\vdots&\ddots&\vdots\\
0&0&0&0&\cdots&1\\
1&0&0&0&\cdots&0
\end{array}.
\right]_{n\times n}
\end{align*}
\end{description}
Then, $\mathbf{I}_{n\times n}+\sum\mathbf{A}_{234\ldots (n+1)1}$ does not have any zero entry. Therefore, by Lemma~\ref{LemmaDecomposabilityOfPermutationMatrices}, $\mathcal{G}$ is indecomposable. 
\end{proof}

\bibliographystyle{amsplain}

\end{document}